\documentclass{amsart}
\usepackage[utf8]{inputenc}
\usepackage[T1]{fontenc}
\usepackage{graphicx}
\usepackage{grffile}
\usepackage{longtable}
\usepackage{wrapfig}
\usepackage{rotating}
\usepackage[normalem]{ulem}
\usepackage{amsmath}
\usepackage{textcomp}
\usepackage{amssymb}
\usepackage{capt-of}
\usepackage{hyperref}
\usepackage{etoolbox}
\usepackage{microtype}
\usepackage{mathrsfs}
\usepackage[english]{babel}
\usepackage{tikz}
\usepackage{tikz-cd}
\usepackage{amsthm}
\usepackage{thmtools}
\usepackage{thm-restate}
\usepackage{xpatch}
\usepackage[backend=biber,style=alphabetic,giveninits=true,url=false,eprint=true,doi=false,isbn=false]{biblatex}
\renewbibmacro{in:}{%
  \ifentrytype{article}{}{\printtext{\bibstring{in}\intitlepunct}}}
\declaretheorem[numberwithin=section]{theorem}  
\declaretheorem[sibling=theorem]{corollary}
\declaretheorem[sibling=theorem]{lemma}
\declaretheorem[sibling=theorem]{proposition}
\declaretheorem[sibling=theorem]{definition}
\declaretheorem[sibling=theorem]{remark}

\DeclareMathOperator\coker{coker}

\DeclareMathOperator\codim{codim}

\DeclareMathOperator\shom{\underline{Hom}}

\newcommand\tensor{\otimes}

\newcommand\dualab\hat
\newcommand\sh\mathscr

\makeatletter
\newcounter{proofstep}
\xpretocmd{\proof}{\setcounter{proofstep}{0}}{}{}
\newcommand{\proofstep}[1]{%
  \par
  \addvspace{\medskipamount}%
  \stepcounter{proofstep}%
  \noindent\emph{Step \theproofstep: #1}\par\nobreak\smallskip
  \@afterheading
}
\makeatother

\tikzset{symbol/.style={draw=none,every to/.append style={edge node={node [sloped, allow upside down, auto=false]{$#1$}}}}}
\title[Chen-Jiang decompositions without Hodge modules]{Chen-Jiang decompositions for projective varieties, without Hodge modules}

\newcommand\vanloc{V}
\newcommand\fmt{\mathsf{FM}}
\newcommand\relgx{\Psi_{X}}
\newcommand\relgxz{\Psi_{X/Z}}
\newcommand\relgxy{\Psi_{X/Y}}

\newcommand\relgxyq{\Psi'_{X/Y}}
\newcommand\relgxzq{\Psi'_{X/Z}}
\newcommand\fibreXy{F}
\newcommand\fibreXz{G}
\newcommand\fibreYz{H}
\newcommand\relgvhs\Phi
\author{Mads Bach Villadsen}
\date{}
\hypersetup{
 pdfauthor={Mads Bach Villadsen},
 pdftitle={Chen-Jiang decompositions for projective varieties, without Hodge modules},
 pdfkeywords={},
 pdfsubject={},
 pdfcreator={Emacs 27.1 (Org mode 9.3.6)}, 
 pdflang={English},
 final,
 bookmarks=true,
 bookmarksnumbered=true,
 bookmarksdepth=subsubsection}
\begin{document}

\maketitle
\begin{abstract}
We give a new proof of a theorem by Pareschi, Popa and Schnell that the direct image of the canonical bundle of a smooth projective variety along a morphism to an abelian variety admits a Chen-Jiang decomposition, without using the theory of Hodge modules.
\end{abstract}

\setcounter{tocdepth}{1}
\tableofcontents

\section{Introduction}
\label{sec:orga5edd2c}
Given a morphism \(f\colon X\to A\) from a smooth projective variety over \(\mathbb{C}\) to an abelian variety, the direct image \(f_{\ast}\omega_X\) is known by work of Green and Lazarsfeld \cite{Green1987} to be a GV-sheaf, that is, the cohomology support locus
\begin{align*}
\vanloc^k(A,f_{\ast}\omega_X)=\{\alpha\in\dualab A\mid H^k(A,f_{\ast}\omega_X\otimes\alpha)\ne 0\}
\end{align*}
has codimension at least \(k\) in the dual abelian variety \(\dualab A\) for each \(k\ge 0\). Moreover, the precise structure of these loci is well understood: The components of \(V^k(A,f_{\ast}\omega_X)\) are translates of abelian subvarieties by \cite{Green1991}, and are in fact translates by points of finite order by work of Simpson \cite{Simpson1993}. We recall these results more precisely in Section \ref{sec:org893a4b0}.

In the case where \(f\) is generically finite, Chen and Jiang \cite{Chen2018} prove a semi-positivity result for \(f_{\ast}\omega_X\) corresponding to the structure of the cohomology support loci. Namely, they prove that there exists a decomposition
\begin{align*}
f_{\ast}\omega_X\cong\bigoplus_i \alpha_i\otimes p_i^{\ast}\sh F_i,
\end{align*}
since called a Chen-Jiang decomposition in \cite{Pareschi2015}, where each \(\alpha_i\in\dualab A\) is a point of finite order, \(p_i\colon A\to A_i\) is a surjective homomorphism of abelian varieties with connected fibres, and each \(\sh F_i\) is an M-regular coherent sheaf on \(A_i\), i.e. for each \(k>0\) we have \(\codim_{\dualab A_i}\vanloc^k(A_i,\sh F_i)>k\). The dual of each \(p_i\) is an inclusion \(\dualab{p_i}\colon\dualab A_i\to\dualab A\), and the codimension \(k\) components of \(V^k(A,f_{\ast}\omega_X)\) for each \(k\) are exactly the translates by \(\alpha_i\) of \(\dualab A_i\), so the failure of \(f_{\ast}\omega_X\) itself to be M-regular is explained by this decomposition. The proof by Chen and Jiang relies on the structural results on \(\vanloc^k(A,f_{\ast}\omega_X)\), but is otherwise algebraic in nature.

Using Hodge modules, this theorem was widely generalized by Pareschi, Popa and Schnell \cite{Pareschi2015}. They prove a Chen-Jiang decomposition result for the associated graded pieces of the Hodge filtration on any polarizable real Hodge module on a compact complex torus. Since direct images of canonical bundles arise in this way, the result of Chen-Jiang is thus extended to arbitrary morphisms (and even to the Kähler setting).

Building on this result, Lombardi, Popa and Schnell \cite{Lombardi2017} prove that direct images of pluricanonical bundles of smooth projective varieties likewise admit Chen-Jiang decompositions, by showing that for \(f\colon X\to A\) and any \(m\ge 2\), there exists a smooth projective variety \(X_{m}\) and a morphism \(f_{m}\colon X_{m}\to A\) such that \(f_{\ast}\omega_X^{\otimes m}\) is a direct summand in \(f_{m\ast}\omega_{X_{m}}\). Jiang and Meng \cite{Jiang2020,Meng2020} independently give Chen-Jiang decompositions for direct images of line bundles equivalent to a rational multiple of the canonical divisor of a klt pair. These results again have applications to the birational theory of irregular varieties.

The proof in \cite{Pareschi2015} relies heavily on Hodge modules and the decomposition theorem, but in the geometric case where \(f\colon X\to A\) is a morphism from a smooth projective variety to an abelian variety, it is reasonable to expect a more direct proof along the lines of the original work by Chen-Jiang. We give such a proof (Section \ref{sec:org337de26}), relying only on the theory of variations of Hodge structure, removing the dependence of the previously mentioned results on Hodge modules.

\begin{restatable}{mainthm}{firstthmA}
For any morphism \(f\colon X\to A\) from a smooth projective variety to an abelian variety, the sheaf \(f_{\ast}\omega_X\) admits a Chen-Jiang decomposition.
\label{thm:A}
\end{restatable}

Following the method of \cite{Chen2018}, the key part of the proof is the following. Assume that the components of the vanishing loci \(\vanloc^k(A,f_{\ast}\omega_X)\) pass through the origin, and are hence abelian subvarieties; this can be arranged via an isogeny of \(A\). Suppose then that \(\dualab B\subset\dualab A\) is a codimension \(k\) component of \(\vanloc^k(A,f_{\ast}\omega_X)\), and let \(p\colon A\to B\) be the projection dual to the inclusion of \(\dualab B\). We must then produce an appropriate M-regular sheaf \(\sh F\) on \(B\) such that \(p^{\ast}\sh F\) is a direct summand of \(f_{\ast}\omega_X\). Given such M-regular sheaves for each such \(B\), the theorem follows for formal reasons (see Lemma \ref{lemma:CJ-decomp-from-individual-decomps}).

In the case where \(f\) is generically finite, Chen and Jiang show that \(R^kp_{\ast}f_{\ast}\omega_X\) is actually the pushforward to \(B\) of the canonical bundle of a lower-dimensional variety, hence admits a Chen-Jiang decomposition by dimensional induction. The M-regular summand of this decomposition serves as \(\sh F\). More precisely, they construct the following diagram (note that the notation here differs slightly from the paper \cite{Chen2018}).

\begin{equation*}
\begin{tikzcd}
X\ar[rd,"g"]\ar[rrd,bend left=20,"f"]\ar[rdd,bend right=20,"q"]\\
&Y\ar[d,"r"]\ar[r,"f'"]&A\ar[d,"p"]\\
&Z\ar[r,"h"]&B
\end{tikzcd}
\end{equation*}

Here \(X\xrightarrow q Z\xrightarrow h B\) is a modified Stein factorization where \(Z\) is smooth and \(h\) generically finite, and \(Y\) is the pullback of \(p\) along \(h\). Then \(q\) is a fibration of relative dimension \(k\), so \(R^kq_{\ast}\omega_X=\omega_Z\) hence \(R^kp_{\ast}f_{\ast}\omega_X=h_{\ast}\omega_Z\). Furthermore \(r\) is a pullback of a morphism of abelian varieties so \(r^{\ast}\omega_Z=\omega_Y\), and since \(g\) is generically finite, \(\omega_Y\) is a direct summand of \(g_{\ast}\omega_X\), hence \(\sh F\) is a direct summand of \(f_{\ast}\omega_X\) by base change.

In the general case of arbitrary \(f\), we use results of Kollár \cite{Kollar1986,Kollar1986a} on variations of Hodge structures and higher direct images of canonical bundles to prove that \(p^{\ast}R^kp_{\ast}f_{\ast}\omega_X\) is a direct summand of \(f_{\ast}\omega_X\), otherwise finishing the proof in the same manner as Chen and Jiang. The more technical proof is deferred to Section \ref{sec:orgd1db5f9}.

\begin{restatable}{mainthm}{firstthmB}
Suppose \(X\xrightarrow f Y\xrightarrow g Z\) are surjective morphisms of smooth projective varieties, and that \(g\) is a smooth fibration of relative dimension \(k\) with \(\omega_{Y/Z}\) trivial. Then \(f_{\ast}\omega_X\) admits \(g^{\ast}R^kg_{\ast}f_{\ast}\omega_X\) as a direct summand.
\label{thm:B-simple}
\end{restatable}

The idea is to construct, by Grothendieck duality, a morphism 
\begin{align*}
\relgxy\colon f_{\ast}\omega_{X/Y}\to g^{\ast}R^kg_{\ast}f_{\ast}\omega_{X/Z}
\end{align*}
which, fibrewise, encodes certain Gysin morphisms. This is most easily described when \(Z\) is a point. Let \(y\in Y\) be a general point and \(\fibreXy=f^{-1}(y)\) the corresponding fibre of \(f\). Then the fibre of \(\relgxy\) at \(y\) is a morphism \(H^0(\fibreXy,\omega_{\fibreXy})\to H^k(Y,f_{\ast}\omega_X)\). When composed with the edge map \(H^k(Y,f_{\ast}\omega_X)\to H^k(X,\omega_X)\) of the Leray spectral sequence we get a morphism \(\relgxyq|_y\colon H^0(\fibreXy,\omega_{\fibreXy})\to H^k(X,\omega_X)\). On the other hand, the inclusion \(\fibreXy\to X\) gives a Gysin morphism \(H^d(\fibreXy,\mathbb{C})\to H^{2k+d}(\fibreXy,\mathbb{C})\) where \(d=\dim \fibreXy\), and the restriction to \(H^0(\fibreXy,\omega_{\fibreXy})\) under the Hodge decomposition coincides with \(\relgxyq|_y\).

To prove the theorem, we construct a morphism of variations of Hodge structure over an open locus which encodes the topological Gysin morphisms on fibres (although for technical reasons, we actually split the direct image \(g_{\ast}\relgxy\) on \(Z\) instead and then use the push-pull adjunction). The splitting then comes from the semisimplicity of the category of polarizable VHS, and we use Kollár's results on higher direct images of canonical bundles \cite[Theorem 2.6]{Kollar1986a} to extend from the open locus. The resulting more precise statement, describing the morphism \(f_{\ast}\omega_X\to g^{\ast}R^kg_{\ast}f_{\ast}\omega_X\), is given as Theorem \ref{thm:B}.

\subsection*{Acknowledgements}
     I would like to thank my advisor Christian Schnell for suggesting this problem, and all his support and advice. I also thank Ben Wu for many useful discussions, as well as comments on a draft of this paper.

\section{Preliminaries}
\label{sec:orgecd4683}
We work throughout with smooth varieties over \(\mathbb{C}\).
\subsection{Generic vanishing}
\label{sec:org893a4b0}
Fix an abelian variety \(A\) throughout this section, and let \(\dualab A\) be the dual abelian variety. Let us recall some basic notions related to GV-sheaves.

For a coherent sheaf \(\sh F\) on \(A\), let \(\vanloc^k(A,\sh F)=\{\alpha\in\dualab A\mid H^k(A,\sh F\otimes\alpha)\ne 0\}\) denote its \(k\)th cohomology support locus. This is a closed subvariety of \(\dualab A\).

\begin{definition}
A coherent sheaf \(\sh F\) on an abelian variety \(A\) is a \emph{GV-sheaf} if \(\codim_{\dualab A}\vanloc^k(A,\sh F)\ge k\) for every \(k\ge 0\), and \emph{M-regular} if \(\codim_{\dualab A}\vanloc^k(A,\sh F)>k\) for every \(k>0\).
\label{def:gv-sheaf}
\end{definition}

Following \cite{Schnell2019}, define the symmetric Fourier-Mukai transform to be the contravariant functor

\[
\fmt_A\colon D^b_{ \mathrm{coh}}(A)\to D^b_{\mathrm{coh}}(\dualab A)
\]

given by the formula

\[
\fmt_A(K)=\mathbf{R}(\mathrm{pr}_2)_{\ast}(P\otimes \mathrm{pr}_1^{\ast}D_A(K))
\]

where \(\mathrm{pr}_1\colon A\times \dualab A\to A\) and  \(\mathrm{pr}_2\colon A\times \dualab A\to \dualab A\) are the two projections, \(P\) is the Poincaré line bundle on \(A\times \dualab A\) normalized by requiring that its fibres over \(0\in A\) respectively \(0\in\dualab A\) are trivial, and 
\[ 
D_A(K)=\mathbf{R}\shom(K,\omega_A[\dim A]) 
\]
is the Grothendieck duality functor on \(A\). Then \(\fmt_A\) is an equivalence of categories with inverse \(\fmt_{\dualab A}\), and is a version of the Fourier-Mukai transform particularly well-adapted to talking about generic vanishing.

GV-sheaves and M-regular sheaves can be defined in terms of the Fourier-Mukai transform. This goes back to Hacon for GV-sheaves, and Pareschi and Popa \cite{Pareschi2011} for M-regular sheaves.

\begin{proposition}[{\cite[Theorem 1.2]{Hacon2004},\cite[Proposition 2.8]{Pareschi2011}}]
A coherent sheaf \(\sh F\) on an abelian variety \(A\) is a GV-sheaf if and only if \(\fmt_A(\sh F)\) is a sheaf (i.e. a complex with cohomology only in degree \(0\)), and \(\sh F\) is M-regular if and only if \(\fmt_A(\sh F)\) is furthermore a torsion-free sheaf.
\label{prop:GV-and-FM}
\end{proposition}

\begin{proposition}[{\cite[Proposition 4.1]{Schnell2019}}]
Let \(f\colon A\to B\) be a morphism of abelian varieties. Denoting by \(\dualab f\colon\dualab B\to\dualab A\) the dual morphism, there are natural isomorphisms
\begin{align*}
\fmt_B\circ \mathbf{R}f_{\ast}&=\mathbf{L}\dualab f^{\ast}\circ\fmt_A\\
\fmt_A\circ \mathbf{L}f^{\ast}&=\mathbf{R}\dualab f_{\ast}\circ\fmt_B
\end{align*}
\label{prop:functoriality-of-FM}
\end{proposition}

\begin{proposition}[{\cite[Proposition 5.1]{Schnell2019}}]
For \(a\in A\) let \(t_a\colon A\to A\) be the translation morphism and \(P_a\) the corresponding line bundle on \(\dualab A\). For \(a\in A\) and \(\alpha\in \dualab A\), there are natural isomorphisms
\begin{align*}
\fmt_A\circ(t_a)_{\ast}=(P_a\otimes -)\circ\fmt_A\\
\fmt_A\circ(P_{\alpha}\otimes -)=(t_{\alpha})_{\ast}\circ\fmt_A
\end{align*}
\label{prop:translation-of-FM}
\end{proposition}

The following characterization of the vanishing loci arising from canonical bundles is due to Green-Lazarsfeld and Simpson for ordinary direct images.
\begin{theorem}[{\cite{Green1987,Green1991,Simpson1993}}]
Suppose \(f\colon X\to A\) is any morphism from a smooth projective variety \(X\) to an abelian variety \(A\). For every \(i\), the sheaf \(R^if_{\ast}\omega_X\) is a GV-sheaf. Furthermore, for every \(k\), every component of \(\vanloc^k(A,R^if_{\ast}\omega_X)\) is a translate of an abelian subvariety of \(\dualab A\) by a point of finite order.
\label{thm:canonical-bundles-are-GV}
\end{theorem}
By work of Kollár, for each \(i\) the higher direct image \(R^if_{\ast}\omega_X\) is a direct summand in \(g_{\ast}\omega_Y\) for some \(g\colon Y\to A\), where \(Y\) is smooth projective \cite[Corollary 2.24]{Kollar1986a}, and the properties stated in the theorem are inherited by direct summands, so we get the theorem for higher direct images as well.

\subsection{Variations of Hodge structure and higher direct images of canonical bundles}
\label{sec:org27b505d}
We will fix notation for, and recall some facts about, variations of Hodge structure. For the full definition see e.g. \cite{Peters2008}. We will follow the notation of \cite{Kollar1986a}; see also that paper for an introduction to canonical extensions.

    The data of a VHS of weight \(k\) on a smooth variety \(U\) consists of a local system \(H\), which for us will always have coefficient group \(\mathbb{Q}\), together with a filtration by holomorphic subbundles of the vector bundle \(\sh H=H\otimes_{\mathbb{Q}}\mathcal{O}_U\), denoted 
\[
\sh H=\sh F^0(H)\supset\cdots\supset\sh F^n(H)\supset 0
\]
by abuse of notation, such that the fibre of the filtration at a given point \(x\) defines a rational Hodge structure of weight \(k\) on the rational vector space \(H_x\). We likewise let \(\operatorname{Gr}^i(H)=\sh F^i/\sh F^{i+1}\), which are again vector bundles. We let \(\nabla\colon \sh H\to\sh H\otimes\Omega_U^1\) denote the induced Gauss-Manin connection, with respect to which the Hodge filtration is required to satisfy the Griffiths transversality condition
\begin{align*}
\nabla(\sh F^p)\subset\sh F^{p-1}\otimes\Omega_U^1
\end{align*}
Finally, a polarization on \(H\) is a map of local systems \(H\otimes H\to \mathbb{Q}_U\) which induces polarizations on the rational Hodge structures \(H_x\) for all \(x\in U\).

Suppose \(f\colon X\to Y\) is a smooth morphism with fibres of dimension \(d\). For any \(k\), the sheaf \(R^kf_{\ast}\mathbb{Q}_X\) is then a local system underlying a polarized VHS of weight \(k\). In particular, in middle degree \(d\) and up the bottom piece of the Hodge filtration is 
\[
\sh F^{k+d}(R^{k+d}f_{\ast}\mathbb{Q}_X)=R^kf_{\ast}\omega_{X/Y}.
\] 

Suppose now that \(H\) underlies a VHS on an open subset \(X^0\subset X\) such that \(X\setminus X^0\) is a normal crossings divisor. Suppose the monodromy of \(H\) around this divisor is quasi-unipotent. A choice of set-theoretic logarithm \(\log\colon \mathbb{C}^{\ast}\to \mathbb{C}\) yields a corresponding canonical extension of \(\sh H\) and the filtration \(\sh F^{\bullet}(H)\) to vector bundles on \(X\). One way to fix such a logarithm is by choosing a fixed length \(2\pi\) interval for the imaginary values of the logarithm. The choice \([0,2\pi)\) is called the upper canonical extension in \cite{Kollar1986a}, and will be denoted by \({}^u\sh H,{}^u\sh F^{\bullet}(H)\). The choice \((-2\pi,0]\) gives the lower canonical extension, denoted by \({}^l\sh H,{}^l\sh F^{\bullet}(H)\). Similarly, \({}^u \operatorname{Gr}^i(H)\) and \({}^l \operatorname{Gr}^i(H)\) denote the associated graded pieces of the extended filtrations.

We briefly recall the local construction. In an analytic neighbourhood of a point in \(X\setminus X^0\), \(X^0\) looks like \((\mathbf{D}^{\ast})^s\times \mathbf{D}^r\) for some \(s\) and \(r\), where \(\mathbf{D}\) is the unit disk and \(\mathbf{D}^{\ast}\) the punctured unit disk. The local monodromy on a fixed fibre of \(H\) in this neighbourhood is generated by the monodromy operators \(T_1,\ldots,T_s\) corresponding to the generators of the fundamental groups of each \(\mathbf{D}^{\ast}\). We will describe how to extend over each \(\mathbf{D}^{\ast}\) separately, so it suffices for us to assume that \(s=1\) and \(r=0\), so we simply have a VHS on \(\mathbf{D}^{\ast}\) with quasi-unipotent monodromy generated by an operator \(T\).

Let \(\exp\colon \mathbf{H}\to \mathbf{D}^{\ast}\) be the universal cover of the unit disk by the left half plane in \(\mathbb{C}\). Then \(T\) acts on global sections of \(\exp^{\ast}\sh H\) by the action of pullback along the translation by \(i\) of \(\mathbf{H}\). In particular, \(T\) acts on the space \(V\) of global flat sections.

Now choose coordinates on \(V\) such that \(T\) decomposes as a product \(T=UD\) of a unipotent matrix \(U\) and a diagonal matrix \(D\). The unipotent matrix has a logarithm given by the usual power series expansion 
\begin{align*}
\log U=\sum_{k=1}^{\infty}(-1)^{k+1}\frac{(B-I)^k}{k},
\end{align*}
while our choice of logarithm gives a logarithm \(\log D\); thus we get \(N=\log T=\log U+\log D\). If \(v\in V\), then the section 
\begin{align*}
s(z)=\exp\left(-\frac{1}{2\pi i}N\cdot\log z\right)v(z)
\end{align*}
is \(T\)-invariant, hence descends to a global section of \(\sh H\). This defines a trivialization of \(\sh H\), hence an extension of \(\sh H\) to a free sheaf on \(\mathbf{D}\). This also gives an extension of the Gauss-Manin connection on \(\sh H\) to a connection on the extension with logarithmic singularities at \(0\), though we will not need a detailed description of this. The nilpotent orbit theorem says that the filtration \(\sh F^{\bullet}(H)\) likewise extends, and these local extensions glue to an extension of \(\sh H\) to \(X\).

One technical obstacle with this theory is that if \(H_i,i=1,2\) are two VHS on \(\mathbf{D}^{\ast}\) with monodromy operators \(T_i\), the monodromy operator of \(H_1\otimes H_2\) is \(T=T_1\otimes T_2\), but the chosen logarithms of \(T\) and the \(T_i\) may not be directly related. Decomposing \(T_i=U_iD_i\) and \(T=UD\) in a unipotent and diagonal part, the eigenvalues of \(D\) are products \(d_1d_2\) where \(d_i\) is an eigenvalue of \(D_i\). But it is not necessarily the case that \(\log (d_1d_2)=\log d_1+\log d_2\), and as a consequence it is not necessarily the case that the canonical extension of \(H\), with this fixed choice of logarithm, is the tensor product of the canonical extensions of the \(H_i\). However, if, say, \(H_1\) has trivial monodromy, then canonical extensions and tensor products will in fact commute in this special case since \(d_1d_2=d_2\) for every pair of eigenvalues as above. This will be important in the proof of Theorem \ref{thm:B-simple}.

Using this machinery of canonical extensions, Kollár proves the following results.
\begin{theorem}[{\cite[Theorem 2.6]{Kollar1986a}}]
Let \(f\colon X\to Y\) be a surjective map of relative dimension \(d\) between smooth projective varieties. Suppose \(Y^0\subset Y\) is an open subset such that \(Y\setminus Y^0\) is a normal crossings divisor and \(f\) is smooth over \(Y^0\). Let \(X^0=f^{-1}(Y^0)\) and \(f^0=f|_{X^0}\). Then
\begin{align*}
R^kf_{\ast}\omega_{X/Y}&\cong{}^u\sh F^{k+d}(R^{k+d}f^0_{\ast}\mathbb{Q}_{X^0})\\
R^kf_{\ast}\mathcal{O}_X&\cong {}^l \operatorname{Gr}^0(R^kf^0_{\ast}\mathbb{Q}_{X^0})
\end{align*}
In particular, \(R^kf_{\ast}\mathcal{O}_X\) and \(R^kf_{\ast}\omega_{X/Y}\) are locally free.
\label{thm:kollarII-2.6}
\end{theorem}

\begin{theorem}[{\cite[Theorem 3.4]{Kollar1986a}}]
Let \(X,Y,Z\) be projective varieties, \(X\) smooth, and \(f\colon X\to Y,g\colon Y\to Z\) surjective maps. Then
\begin{enumerate}
\item \(R^p(g\circ f)_\ast\omega_X\cong\bigoplus_i R^ig_\ast R^{p-i}f_\ast\omega_X\);
\item \(R^ig_{\ast}R^jf_{\ast}\omega_X\) is torsion-free;
\item \(R^ig_{\ast}R^jf_{\ast}\omega_X=0\) if \(i>\dim Y-\dim Z\);
\item In the derived category of coherent sheaves on \(Z\),
\end{enumerate}
\begin{align*}
\mathbf{R}g_{\ast}R^jf_{\ast}\omega_X=\bigoplus_iR^ig_{\ast}R^jf_{\ast}\omega_X[-i].
\end{align*}
\label{thm:kollarII-3.4}
\end{theorem}

Theorem \ref{thm:kollarII-2.6} shows that \(R^kf_{\ast}\omega_{X/Y}\), hence also \(R^kf_{\ast}\omega_X\), is globally controlled by a polarized VHS on an open subset. Observe the following.
\begin{enumerate}
\item The formation of canonical extensions is compatible with taking direct sums of VHS.
\item The open subset \(Y^0\subset Y\) in the theorem does not have to be the entire smooth locus of \(f\); any non-empty Zariski-open subset thereof suffices as long as the complement is normal crossings.
\item The category of polarizable VHS is semisimple \cite[Theorem 10.13]{Peters2008}.
\end{enumerate}
It follows that one way to get a direct sum decomposition of \(R^kf_{\ast}\omega_X\) is to construct an appropriate morphism involving the VHS \(R^{k+d}f^0_{\ast}\mathbb{Q}_{X^0}\), for some appropriate \(Y^0\) as in the theorem. This will be the mechanism for getting the splitting in Theorem \ref{thm:B-simple}.

\subsection{Chen-Jiang decompositions}
\label{sec:orgec703c2}
Let's recall the following definition and proposition from \cite{Lombardi2017}.

\begin{definition}[{\cite[Definition 4.1]{Lombardi2017}}]
Suppose \(\sh{F}\) is a coherent sheaf on an abelian variety \(A\). A Chen-Jiang decomposition of \(\sh{F}\) is a direct sum decomposition
\[
\sh F\cong\bigoplus_i \alpha_i\otimes p_i^{\ast}\sh{F}_i
\]
where each \(p_i\colon A\to A_i\) is a surjective homomorphism of abelian varieties with connected fibres, each \(\sh F_i\) is an M-regular coherent sheaf on \(A_i\) and each \(\alpha_i\in\dualab A\) is a line bundle of  finite order.
\label{def:CJ-decomp}
\end{definition}

\begin{proposition}[{\cite[Proposition 4.6]{Lombardi2017}}]
Suppose \(\sh F'\) and \(\sh F''\) are coherent sheaves on an abelian variety \(A\). If \(\sh F'\oplus\sh F''\) admits a Chen-Jiang decomposition, so do \(\sh F'\) and \(\sh F''\).
\label{prop:CJ-decomp-direct-summand}
\end{proposition}

The following proposition is essentially proven as part of the proof of \cite[Theorem 3.5]{Chen2018}.
\begin{proposition}
Suppose \(\sh F\) is a coherent sheaf on an abelian variety \(A\), and \(\phi\colon A'\to A\) is an isogeny. Then \(\sh F\) admits a Chen-Jiang decomposition if and only if \(\phi^{\ast}\sh F\) does.
\label{prop:CJ-decomp-isogeny}
\end{proposition}
\begin{proof}
If \(\sh F\) admits a Chen-Jiang decomposition then clearly so does \(\phi^{\ast}\sh F\).

In the other direction note that by Propositions \ref{prop:GV-and-FM}, \ref{prop:functoriality-of-FM}, and \ref{prop:translation-of-FM}, a Chen-Jiang decomposition of \(\sh \phi^{\ast}F\) is equivalent to a decomposition
\begin{align*}
\fmt_{A'}(\phi^{\ast}\sh F)\cong\bigoplus_i \tau_{\alpha_{i}\ast}\iota_{i\ast}\sh G_i
\end{align*}
where for each \(i\), \(\tau_{\alpha_i}\) is a translation of \(A'\) by a point \(\alpha_i\) of finite order, \(\iota_i\colon\dualab A_i\to\dualab A'\) the inclusion of an abelian subvariety, and \(\sh G_i\) is a torsion free sheaf on \(\dualab A_i\).

Now for each \(i\), \(\dualab\phi^{\ast}\tau_{\alpha_i\ast}\iota_{i\ast}\sh G_i\) is the direct image of a torsion free sheaf on \(\dualab\phi^{-1}(\dualab A_i)\), which is again a torsion translate of an abelian subvariety of \(\dualab A\); namely the direct image of \(\left(\dualab\phi|_{\dualab\phi^{-1}(\dualab A_i)}\right)^{\ast}\sh G_i\) translated by a preimage of \(\alpha_i\).

By Proposition \ref{prop:functoriality-of-FM} and since \(\phi\) is an isogeny we have 
\[\dualab\phi^{\ast}\fmt_{A'}(\phi^{\ast}\sh F)=\fmt_A(\phi_{\ast}\phi^{\ast}\sh F),
\]
 so \(\phi_{\ast}\phi^{\ast}\sh F\) admits a Chen-Jiang decomposition. But \(\sh F\) is a direct summand thereof, hence admits a Chen-Jiang decomposition by Proposition \ref{prop:CJ-decomp-direct-summand}.
\end{proof}

Given a morphism \(f\colon X\to A\) to an abelian variety, we will need to understand how the image \(f(X)\) relates to the various components of the cohomology support loci of \(f_{\ast}\omega_X\).

\begin{lemma}
Suppose given \(f\colon X\to A\) where \(X\) is a smooth projective variety, \(A\) an abelian variety, and suppose \(\dualab B\subset\dualab A\) is a codimension \(k\) component of \(\vanloc^k(A,f_{\ast}\omega_X)\) which passes through \(0\in \dualab A\) (and is hence an abelian subvariety by Theorem \ref{thm:canonical-bundles-are-GV}). Let \(p\colon A\to B\) be dual to the inclusion. Then all fibres of \(f(X)\) over \(B\) are of dimension \(k\), hence \(f(X)\) is the preimage of \(p(f(X))\). In particular \(p|_{f(X)}\colon f(X)\to p(f(X))\) is a smooth fibration with trivial relative canonical bundle.
\label{lemma:f(X)-is-abelian-fibration}
\end{lemma}
\begin{proof}
Observe that \(p\) is smooth of relative dimension \(k\), so it suffices to show that a general fibre of \(p|_{f(X)}\) has dimension \(k\). Suppose \(\beta\in\dualab B\). By Kollár's result (Theorem \ref{thm:kollarII-3.4}), we have
\begin{align*}
h^k(A,f_{\ast}\omega_X\otimes p^{\ast}\beta)&=\sum_{i=0}^kh^i(B,R^{k-i}p_{\ast}f_{\ast}\omega_X\otimes\beta)
\end{align*}
The left hand side is non-zero by the assumptions on \(B\), while for general \(\beta\in\dualab B\), the terms with \(i>0\) on the right hand side vanish since \(R^{k-i}p_{\ast}f_{\ast}\omega_X\) is a direct summand of \(R^{k-i}(p\circ f)_{\ast}\omega_X\) (by Theorem \ref{thm:kollarII-3.4} again), which is a GV-sheaf by Theorem \ref{thm:canonical-bundles-are-GV}. It follows that \(h^0(B,R^kp_{\ast}f_{\ast}\omega_X\otimes\beta)\) is non-zero, hence that \(R^kp_{\ast}f_{\ast}\omega_X\) is non-zero.

To conclude, recall that \(R^k(p\circ f)_{\ast}\omega_X\), hence the summand \(R^kp_{\ast}f_{\ast}\omega_X\), is torsion-free over the image of \(f\) in \(B\) (Theorem \ref{thm:kollarII-3.4}). By base change over the smooth locus of \(f\), this \(k^{\text{th}}\) higher direct image would vanish if the general fibre had dimension smaller than \(k\), hence the general fibre actually has dimension \(k\).
\end{proof}
\subsection{Generic base change}
\label{sec:org1115af6}
Finally we will need the following generic base change theorem. Suppose \(X\xrightarrow f Y\xrightarrow g Z\) are proper morphisms of schemes of finite type over a field, that \(Z\) is generically reduced, and that \(h=g\circ f\) is surjective. Let \(\sh F\) be a coherent sheaf on \(X\). For \(z\in Z\), let \(\fibreXz=h^{-1}(z)\) and \(\fibreYz=g^{-1}(z)\)  be the fibres over \(z\), and consider \(f|_{\fibreXz}\colon\fibreXz\to\fibreYz\).
\begin{proposition}[{\cite[Proposition 5.1]{Lombardi2017}}]
In the setting above, there is a non-empty Zariski-open subset \(U\subset Z\) such that the base change morphism
\begin{align*}
\left(R^if_{\ast}\sh F\right)\!|_{\fibreYz}\to R^i(f|_{\fibreXz})_{\ast}\!\left(\sh F|_{\fibreXz}\right)
\end{align*}
is an isomorphism of sheaves on \(\fibreYz\) for every \(z\in U\) and every \(i\).
\label{prop:generic-base-change}
\end{proposition}
In particular, if \(X,Y\) and \(Z\) are smooth projective and \(\sh F\) is the relative canonical bundle \(\omega_{X/Y}=\omega_X\otimes f^{\ast}\omega_Y^{-1}\), this says that the restriction \((f_{\ast}\sh F)|_{\fibreYz}\) is, for general \(z\in Z\), isomorphic to the relative canonical bundle \(\omega_{\fibreXz/\fibreYz}\) of the morphism \(f|_{\fibreXz}\), and similarly for the higher direct images.

\section{Chen-Jiang decompositions for direct images of canonical bundles}
\label{sec:org337de26}
The goal of this section is to prove that Chen-Jiang decompositions always exist for direct images of canonical bundles, following the approach originally used in \cite{Chen2018} to give the decompositions for generically finite morphisms.

Let us first recall the original proof of \cite[Theorem 3.5]{Chen2018}, namely that if \(f\colon X\to A\) is a generically finite morphism to an abelian varietythen \(f_{\ast}\omega_X\) admits a Chen-Jiang decomposition. First, by Theorem \ref{thm:canonical-bundles-are-GV} and Proposition \ref{prop:CJ-decomp-isogeny} it suffices to assume that all components of the cohomology support loci \(\vanloc^k(A,f_{\ast}\omega_X)\) for every \(k\) passes through the origin of \(A\). Indeed choose a finite order point in each such component, then choose an isogeny \(\phi\colon A'\to A\) such that each of those finite points get mapped to the origin of \(\dualab{A'}\) under \(\phi^{\ast}\). If \(X'=X\times_AA'\) and \(f'\colon X'\to A'\) is the second projection, then \(\phi^{\ast}f_{\ast}\omega_X=f'_{\ast}\omega_{X'}\), and the components of \(\vanloc^k(A',f'_{\ast}\omega_{X'})\) all pass through the origin of \(A'\).

Chen and Jiang's proof of  \cite[Theorem 3.4]{Chen2018} is then essentially to show the following result, and prove that the conditions are satisfied when \(f\) is generically finite.
\begin{lemma}
Assume all components of each \(\vanloc^k(A,f_{\ast}\omega_X)\) pass through \(0\in A\). Suppose that for each \(k>0\), and for each component \(\dualab{B}\) of \(\vanloc^k(A,f_{\ast}\omega_X)\) of codimension \(k\), there exists an M-regular sheaf \(\sh F_B\) on \(B\) with the following properties.
\begin{enumerate}
\item If \(p_B\colon A\to B\) is dual to the inclusion \(\dualab{B}\to\dualab{A}\), then \(f_{\ast}\omega_X\) admits \(p_B^{\ast}\sh F_B\) as a direct summand.
\item For general \(\beta\in\dualab{B}\), \(h^k(A,f_{\ast}\omega_X\otimes p_B^{\ast}\beta)=h^0(B,\sh F_B\otimes \beta)\)
\end{enumerate}
Then \(f_{\ast}\omega_X\) admits a Chen-Jiang decomposition.
\label{lemma:CJ-decomp-from-individual-decomps}
\end{lemma}

\begin{proof}
Following the notation of \cite{Chen2018}, let \(S^k_X\) denote the set of codimension  \(k\) components of \(\vanloc^k(A,f_{\ast}\omega_X)\). Then Step 2 of the proof of \cite[Theorem 3.4]{Chen2018} applies verbatim to show that there exists a decomposition
\[
f_{\ast}\omega_X\cong \sh W\oplus\bigoplus_{k>0,\dualab B\in S^k_X}p_B^{\ast}\sh F_B
\]
It remains to show that \(\sh W\) is M-regular. This follows from the arguments of Step 3 of the proof of \cite[Theorem 3.4]{Chen2018} and the second point in the statement of this lemma.
\end{proof}

As outlined in the introduction, Chen and Jiang critically use the fact that if \(f\colon X\to Y\) is generically finite and surjective, then \(f_{\ast}\omega_X\) is a direct summand of \(\omega_Y\). The main new result is Theorem \ref{thm:B-simple}, which serves as a generalization of this statement to arbitrary morphisms. We defer the proof to Section \ref{sec:orgd1db5f9}, but recall the statement here.

\firstthmB*

Recall the statement of Theorem \ref{thm:A} as well.

\firstthmA*

\begin{corollary}
Suppose \(f\colon X\to A\) is a morphism from a smooth projective variety \(X\) to an abelian variety \(A\). Then \(R^if_{\ast}\omega_X\) admits a Chen-Jiang decomposition on \(A\) for all \(i\).
\label{cor:thm:A}
\end{corollary}
\begin{proof}
For \(i>0\), there exists by work of Kollár \cite[Corollary 2.24]{Kollar1986a} a smooth variety \(Z\) with \(\dim Z=\dim X-i\) and a morphism \(\phi\colon Z\to A\) such that \(R^if_{\ast}\omega_X\) is a direct summand of \(\phi_{\ast}\omega_Z\) (the claim about the dimension of \(Z\) follows from the proof of Kollár's corollary; in fact \(Z\) is a generic intersection of \(i\) hyperplane sections of some birational model of \(X\)). The result follows by Proposition \ref{prop:CJ-decomp-direct-summand}.
\end{proof}

\begin{remark}
The proof of Corollary \ref{cor:thm:A} for a fixed \(X\) and \(i>0\) only relies on Theorem \ref{thm:A} in the case of varieties with strictly smaller dimension than \(X\). In the proof of Theorem \ref{thm:A}, we can thus assume that \(R^if_{\ast}\omega_X\) admits a Chen-Jiang decomposition for all \(i>0\) by induction on \(\dim X\).
\label{rem:thm:A}
\end{remark}

\begin{proof}[{Proof of Theorem \ref{thm:A}}]
By Theorem \ref{thm:canonical-bundles-are-GV} and Proposition \ref{prop:CJ-decomp-isogeny}, we can assume that all components of the cohomology support loci \(\vanloc^k(A,f_{\ast}\omega_X)\) are in fact abelian subvarieties. It suffices to verify the conditions of Lemma \ref{lemma:CJ-decomp-from-individual-decomps}. Namely, suppose \(\dualab B\) is a codimension \(k\) component of \(\vanloc^k(A,f_{\ast}\omega_X)\) with \(k>0\), and \(p\colon A\to B\) is dual to the inclusion \(\dualab B\subset\dualab A\). Then we must show that there exists an M-regular sheaf \(\sh F\) on \(B\) such that \(h^k(A,f_{\ast}\omega_X\otimes p^{\ast}\beta)=h^0(B,\sh F\otimes \beta)\) for general \(\beta\in\dualab B\), and that \(f_{\ast}\omega_X\) admits \(p^{\ast}\sh F\) as a direct summand.

Let \(Y\subset A\) be the image of \(f\), and \(p(Y)=Z\subset B\).  Let \(g\colon Y\to Z\) denote the restriction of \(p\) and \(q=g\circ f\). By Lemma \ref{lemma:f(X)-is-abelian-fibration}, \(g\) is a smooth fibration of relative dimension \(k\), and \(\omega_{Y/Z}\) is trivial. Now let \(\pi_Z\colon Z'\to Z\) be a resolution of singularities of \(Z\), and construct the following diagram by pullback.

\begin{equation*}
\begin{tikzcd}
X'\ar[r,"f'"]\ar[d,"\pi_X"]&Y'\ar[r,"g'"]\ar[d,"\pi_Y"]&Z'\ar[d,"\pi_Z"]\\
X\ar[r,"f"]&Y\ar[r,"g"]&Z
\end{tikzcd}
\end{equation*}

Then \(g'\) is again a smooth fibration of relative dimension \(k\), and \(\omega_{Y'/Z'}\) is trivial. By Theorem \ref{thm:B-simple} applied to the sequence \(X'\xrightarrow {f'} Y'\xrightarrow {g'} Z'\), we can now conclude that \(f'_{\ast}\omega_{X'}\) admits \(g'^{\ast}R^kg'_{\ast}f'_{\ast}\omega_{X'}\) as a direct summand. Pushing forward to \(Y\) it follows that \(f_{\ast}\omega_X\) admits \(g^{\ast}R^kg_{\ast}f_{\ast}\omega_X\) as a direct summand, by flat base change along \(g\) and the fact that \(\pi_{X\ast}\omega_{X'}=\omega_X\) since \(\pi_X\) is birational.

To apply Lemma \ref{lemma:CJ-decomp-from-individual-decomps}, it then suffices to show that \(R^kg_{\ast}f_{\ast}\omega_X\) admits as direct summand an M-regular sheaf \(\sh F\) such that \(h^0(B,R^kp_{\ast}f_{\ast}\omega_X\otimes \beta)=h^0(B,\sh F\otimes \beta)\) for general \(\beta\in\dualab B\). It suffices to show that \(R^kg_{\ast}f_{\ast}\omega_X\) admits a Chen-Jiang decomposition as a sheaf on \(B\), as we can then take \(\sh F\) to be the M-regular summand of the decomposition.

But \(R^kg_{\ast}f_{\ast}\omega_X\) is a direct summand of \(R^kq_{\ast}\omega_X\) by Kollár's result (Theorem \ref{thm:kollarII-3.4}). By dimensional induction and Corollary \ref{cor:thm:A} (see Remark \ref{rem:thm:A}), \(R^kq_{\ast}\omega_X\) admits a Chen-Jiang decomposition on \(B\), hence so does \(R^kg_{\ast}f_{\ast}\omega_X\) by Proposition \ref{prop:CJ-decomp-direct-summand}.
\end{proof}
\section{Splitting of direct images of canonical bundles}
\label{sec:orgd1db5f9}
The goal of this section is to prove Theorem \ref{thm:B-simple}. More precisely, suppose given surjective morphisms of smooth varieties \(X\xrightarrow f Y\xrightarrow g Z\) where \(g\) is flat. Let \(q=g\circ f\). We will construct a morphism \(\relgx \colon f_{\ast}\omega_X\to g^{!}R^kg_{\ast}f_{\ast}\omega_X[-k]\) where \(k=\dim Y-\dim Z\); as \(g\) is flat, this is actually a map of sheaves. We will then show that this morphism is split surjective in the setting of Theorem \ref{thm:B-simple}, using the results by Kollár outlined in Section \ref{sec:org27b505d}.

\subsection{Relative Gysin morphism for canonical bundles}
\label{sec:org67f5d87}
       To construct the desired morphism, note first that \(R^ig_{\ast}f_{\ast}\omega_X\) vanishes for \(i>k\) by Kollár's result (Theorem \ref{thm:kollarII-3.4}), hence \(\mathbf{R}g_{\ast}f_{\ast}\omega_X\), as an object of the derived category of coherent sheaves on \(Z\), is concentrated in degrees \(0\) to \(k\). Thus the projection to the \(k^{\text{th}}\) cohomology sheaf gives a map \(\mathbf{R}g_{\ast}f_{\ast}\omega_X\to R^kg_{\ast}f_{\ast}\omega_X[-k]\). By adjunction this corresponds to a morphism 
\[
\relgx \colon f_{\ast}\omega_X\to g^{!}R^kg_{\ast}f_{\ast}\omega_X[-k].
\]
Since \(Y\) and \(Z\) are smooth and \(g\) is flat of relative dimension \(k\), we have 
\[
g^{!}R^kg_{\ast}f_{\ast}\omega_X[-k]\cong \omega_{Y/Z}\otimes g^{\ast}R^kg_{\ast}f_{\ast}\omega_X.
\]

Note that there's a canonical morphism \(R^kg_{\ast}f_{\ast}\omega_X\to R^kq_{\ast}\omega_X\), namely the edge map from the composed functor spectral sequence \(R^ig_{\ast}R^jf_{\ast}\omega_X\implies R^{i+j}q_{\ast}\omega_X\). By Kollár's result (Theorem \ref{thm:kollarII-3.4}), this map is an inclusion of a direct summand. Pulling this back to \(Y\), and applying twists by canonical bundles and composing with twists of \(\relgx\), yields the following morphisms.

\begin{equation*}
\begin{tikzcd}
f_{\ast}\omega_X\ar[r,"\relgx"]&\omega_{Y/Z}\otimes g^{\ast}R^kg_{\ast}f_{\ast}\omega_X \\
f_{\ast}\omega_{X/Z}\ar[r,"\relgxz"]\ar[rr,bend right=15,"\relgxzq"]&\omega_{Y/Z}\otimes g^{\ast}R^kg_{\ast}f_{\ast}\omega_{X/Z}\ar[r]&\omega_{Y/Z}\otimes g^{\ast}R^kq_{\ast}\omega_{X/Z}\\
f_{\ast}\omega_{X/Y}\ar[r,"\relgxy"]\ar[rr,bend right=15,"\relgxyq"]&g^{\ast}R^kg_{\ast}f_{\ast}\omega_{X/Z}\ar[r]&g^{\ast}R^kq_{\ast}\omega_{X/Z}
\end{tikzcd}
\end{equation*}

\begin{lemma}
For a general point \(y\in Y\), let \(\fibreXy=f^{-1}(y),\fibreXz=q^{-1}(g(y))\) and \(\fibreYz=g^{-1}(g(y))\). 
\begin{enumerate}
\item The fibre \(\relgxyq|_y\colon H^0(\fibreXy,\omega_{\fibreXy})\to H^k(\fibreXz,\omega_{\fibreXz})\) of \(\relgxyq\) is the Gysin morphism of the inclusion \(\fibreXy\subset\fibreXz\).
\item The fibre of \(\relgxy\) at \(y\) is a surjective morphism
\end{enumerate}
\[
\relgxy|_y\colon H^0(\fibreXy,\omega_{\fibreXy})\to H^k(\fibreYz,(f|_{\fibreXz})_{\ast}\omega_{\fibreXz}).
\]
\begin{enumerate}
\item Furthermore, let \(z=g(y)\) and assume \(g\) is a fibration. Then the fibre \((g_{\ast}\relgxz)|_z\) of \(g_{\ast}\relgxy\colon q_{\ast}\omega_{X/Y}\to R^kg_{\ast}f_{\ast}\omega_{X/Z}\) at \(z\) is an isomorphism for general \(y\).
\end{enumerate}
\label{lemma:fibres-of-G}
\end{lemma}
The notation can be summarized in the following commuting diagram, where all squares are cartesian.

\begin{equation*}
\begin{tikzcd}
\fibreXy\ar[r,hook]\ar[d]&\fibreXz\ar[r,hook]\ar[d]&X\ar[d]\\
\{y\}\ar[r,hook]&\fibreYz\ar[r,hook]\ar[d]&Y\ar[d]\\
&\{z\}\ar[r,hook]&Z
\end{tikzcd}
\end{equation*}

\begin{proof}
\proofstep{Identifying fibres of \(\relgxyq\) with Gysin morphisms.}

By generic base change (Proposition \ref{prop:generic-base-change}), we can assume \(Z\) is a point, so \(\fibreXz=X,\fibreYz=Y\) and \(\fibreXy\) is a general fibre of \(f\). Then \(\dim Y=k,\dim X=d+k\), and \(\dim \fibreXy=d\). 

Taking a log resolution of \(Y\), we can furthermore assume that the the discriminant locus of \(f\) is normal crossings. By Kollár's result (Theorem \ref{thm:kollarII-2.6}) all the higher direct images \(R^if_{\ast}\omega_X\) are then locally free. At a general \(y\in Y\), we get a sequence 
\[
H^0(\fibreXy,\omega_{\fibreXy})\to H^k(Y,f_{\ast}\omega_X)\to H^k(X,\omega_X),
\]
and the linear dual is a sequence 
\[
H^d(X,\mathcal{O}_X)\to H^0(Y,R^df_{\ast}\mathcal{O}_X)\to H^d(\fibreXy,\mathcal{O}_{\fibreXy})
\]
by Serre duality on the spaces involved. We claim that the first morphism is the edge map from the second page of the Leray spectral sequence for \(\mathcal{O}_X\), the second the base change morphism to the fibre, and the composition as a result the restriction to a fibre.

     Let 
\[
D_Y(-)=\mathbf{R}\shom(-,\omega_Y[\dim Y])
\]
be the Serre duality functor for \(Y\) and similarly \(D_X\) for \(X\). Then 
\[
D_Y(\mathbf{R}f_{\ast}\omega_X)=\mathbf{R}f_{\ast}\mathcal{O}_X[\dim X],
\]
and since the discriminant locus of \(f\) is a normal crossings divisor, the \(R^if_{\ast}\omega_X\) are locally free and 
\[
D_Y(R^if_{\ast}\omega_X[-i])=R^{d-i}f_{\ast}\mathcal{O}_X[i+\dim Y],
\]
recalling that \(d=\dim X-\dim Y\). 

     To compute the dual of the fibre 
\[
\relgxy|_y\colon H^0(\fibreXy,\omega_{\fibreXy})\to H^k(Y,f_{\ast}\omega_X),
\]
we can apply \(\shom(-,\mathcal{O}_Y)\) and compute fibres of the resulting morphism, since the sheaves involved are locally free. But since \(\relgxy=\relgx\otimes\omega_Y^{-1}\) we have 
\[
\shom(\relgxy,\mathcal{O}_Y)=D_Y(\relgx)[-\dim Y].
\]

     Now the morphism \(\relgx\) is constructed as the composition
\[
f_{\ast}\omega_X\to g^{!}\mathbf{R}g_{\ast}f_{\ast}\omega_X\to g^{!}R^kg_{\ast}f_{\ast}\omega_X[-k]
\]
where the first morphism is the unit of adjunction, and the second is the projection to the highest cohomology sheaf. The (Serre) dual of the former is the counit of adjunction 
\[
g^{\ast}\mathbf{R}g_{\ast}R^df_{\ast}\mathcal{O}_X\to R^df_{\ast}\mathcal{O}_X,
\]
while the dual of the latter is the inclusion 
\[
g_{\ast}R^df_{\ast}\mathcal{O}_X\to \mathbf{R}g_{\ast}R^df_{\ast}\mathcal{O}_X
\]
of the lowest direct image along \(g\).  The composition is thus just the counit of the non-derived adjunction, namely \(g^{\ast}g_{\ast}R^df_{\ast}\mathcal{O}_X\to R^df_{\ast}\mathcal{O}_X\). Since \(Z\) is a point, \(g^{\ast}g_{\ast}R^df_{\ast}\mathcal{O}_X=H^0(Y,R^df_{\ast}\mathcal{O}_X)\otimes \mathcal{O}_Y\), and the fibre at a general \(y\in Y\) is just given by restricting a global section to the fibre over \(y\), so \(H^0(Y,R^df_{\ast}\mathcal{O}_X)\to H^d(\fibreXy,\mathcal{O}_{\fibreXy})\) in the sequence above is given as claimed.

     The morphism 
\[
H^k(Y,f_{\ast}\omega_X)\to H^k(X,\omega_X)
\]
is an edge map in the Leray spectral sequence for \(\omega_X\) with respect to \(f\), induced by the canonical map \(f_{\ast}\omega_X\to \mathbf{R}f_{\ast}\omega_X\). The edge map
\[
H^d(X,\mathcal{O}_X)\to H^0(Y,R^df_{\ast}\mathcal{O}_X)
\]
in the Leray spectral sequence for \(\mathcal{O}_X\) is similarly induced by the projection \(\mathbf{R}f_{\ast}\mathcal{O}_X\to R^df_{\ast}\mathcal{O}_X[-d]\) to the highest direct image (the ones in degree \(>d\) vanishing by duality). We claim that these maps get identified under \(D_Y\).

Let \(K^0\xrightarrow{d}K^1\to\cdots\) be any locally free resolution of \(\mathbf{R}f_{\ast}\omega_X\). Then the inclusion \(f_{\ast}\omega_X\to \mathbf{R}f_{\ast}\omega_X\) is canonically identified with the inclusion \(\ker d\to K^0\). Applying \(D_Z\) (and dropping the index shifts from the notation) gives a resolution \(\left(\cdots\to K^{1\vee}\to K^{0\vee}\right)\otimes\omega_Y\) of \(\mathbf{R}f_{\ast}\mathcal{O}_X\). The inclusion \(\ker d\to K^0\) gets mapped under \(D_Y\) to the surjection \(\left(K^{0\vee}\to\coker d^{\vee}\right)\otimes\omega_Y\). But this is just the canonical map \(\mathbf{R}f_{\ast}\mathcal{O}_X\to R^df_{\ast}\mathcal{O}_X[-d]\) as desired. This proves the claim that \(\relgxyq\) is the Gysin morphism on general fibres.

\proofstep{Surjectivity of general fibres of \(\relgxy\).}

Let us now show that the restriction 
\[
H^0(Y,R^df_{\ast}\mathcal{O}_X)\to H^d(\fibreXy,\mathcal{O}_{\fibreXy})
\]
is injective. Suppose that \(\alpha\in H^0(Y,R^df_{\ast}\mathcal{O}_X)\) vanishes when restricted to some point \(y_0\) in the smooth locus \(Y^0\subset Y\) of \(f\). Lift \(\alpha\) to an element \(\tilde{\alpha}\) of \(H^d(X,\mathcal{O}_X)\), and let \(X^0=f^{-1}(Y^0)\) and \(f^0\colon X^0\to Y^0\) be the restriction of \(f\). Since \(f^0\) is smooth, \(R^df^0_{\ast}\mathbb{C}_{X^0}\) is a local system, and the Leray spectral sequence for \(\mathbb{C}_{X^0}\) with respect to \(f^0\) degenerates, we get a map 
\[
\pi\colon H^d(X,\mathbb{C}_X)\to H^0(Y^0,R^df^0_{\ast}\mathbb{C}_{X^0}).
\]
Furthermore, the fibre of \(R^df^0_{\ast}\mathbb{C}_{X^0}\) at \(y\in Y^0\) is exactly \(H^d(\fibreXy,\mathbb{C})\). Applying the Hodge decomposition for \(X\) and \(\fibreXy\), the restriction of \(\pi(\tilde{\alpha})\) to \(y\) equals the restriction of \(\alpha\) to \(y\), which vanishes for \(y=y_0\). But \(\pi(\tilde{\alpha})\) is then a section of a local system which vanishes at a point, and since \(Y^0\) is connected, \(\pi(\hat{\alpha})\) must thus be identically \(0\). It follows that \(\alpha\) vanishes at every \(y\in Y^0\). As \(R^dy_{\ast}\mathcal{O}_X\) is locally free, and since \(\alpha\) vanishes on a dense open set, we get \(\alpha=0\). This dually gives the desired surjectivity.

\proofstep{Surjectivity of general fibres of \(g_\ast\relgxy\).}

For the final statement of the lemma, consider the monodromy action of \(\pi_1(Y^0)\) on \(H^d(\fibreXy,\mathbb{C})\), and the subspace \(H^d(\fibreXy,\mathbb{C})^{\pi_1(Y^0)}\) of invariants under this action. 

Define \(H^0(\fibreXy,\omega_{\fibreXy})^{\pi_1(Y^0)}\) as the preimage of \(H^{d}(\fibreXy,\mathbb{C})^{\pi_1(Y^0)}\) under the inclusion \(H^0(\fibreXy,\omega_{\fibreXy})\hookrightarrow H^d(\fibreXy,\mathbb{C})\). Note that \(\pi_1(Y^0)\) does not act on \(H^0(\fibreXy,\omega_{\fibreXy})\); we are considering invariants under the action on the larger space \(H^d(\fibreXy,\mathbb{C})\).

Consider the following commuting diagram.

\begin{equation*}
\begin{tikzcd}
  H^0(Y,f_{\ast}\omega_{X/Y})\ar[r,"H^0(\relgxy)"]\ar[d]&H^k(Y,f_{\ast}\omega_X)\\
  H^0(\fibreXy,\omega_{\fibreXy})\ar[ur,swap,"\relgxy|_y"]\ar[r,symbol=\supseteq]&H^0(\fibreXy,\omega_{\fibreXy})^{\pi_1(Y^0)}\ar[u]
\end{tikzcd}
\end{equation*}
We claim that the right side vertical map is an isomorphism, while the image of the restriction morphism \(H^0(Y,f_{\ast}\omega_{X/Y})\to H^0(\fibreXy,\omega_{\fibreXy})\) contains \(H^0(\fibreXy,\omega_{\fibreXy})^{\pi_1(Y^0)}\); this would yield the desired surjectivity.

Define dually \(H^d(\fibreXy,\mathcal{O}_{\fibreXy})^{\pi_1(Y^0)}\subset H^d(\fibreXy,\mathcal{O}_{\fibreXy})\) as the image of \(H^d(\fibreXy,\mathbb{C})^{\pi_1(Y^0)}\) under the canonical projection \(H^d(\fibreXy,\mathbb{C})\to H^d(\fibreXy,\mathcal{O}_{\fibreXy})\). Note again that \(\pi_1(Y^0)\) does not act on \(H^d(\fibreXy,\mathcal{O}_\fibreXy)\) by itself, only on the larger \(H^d(\fibreXy,\mathbb{C})\).

Then \(H^d(\fibreXy,\mathcal{O}_{\fibreXy})^{\pi_1(Y^0)}\) is exactly the image of the restriction morphism 
\[
H^0(Y,R^df_{\ast}\mathcal{O}_X)\to H^d(\fibreXy,\mathcal{O}_{\fibreXy}).
\]
Indeed the image of the restriction map \(H^d(X,\mathcal{O}_X)\to H^d(\fibreXy,\mathcal{O}_\fibreXy)\) is exactly \(H^d(\fibreXy,\mathcal{O}_{\fibreXy})^{\pi_1(Y^0)}\) by the global invariant cycles theorem and the fact that the restriction map in singular cohomology is a morphism of Hodge structures, and the coherent restriction map factors through \(H^0(Y,R^df_{\ast}\mathcal{O}_X)\).

     We conclude that the restriction morphism gives an isomorphism 
\[
H^0(Y,R^df_{\ast}\mathcal{O}_X)\xrightarrow{\sim} H^d(\fibreXy,\mathcal{O}_{\fibreXy})^{\pi_1(Y^0)}
\]
by the injectivity from the previous step of this proof. Since \(H^d(\fibreXy,\mathcal{O}_\fibreXy)^{\pi_1(Y^0)}\) and \(H^0(\fibreXy,\omega_{\fibreXy})^{\pi_1(Y^0)}\) are dual, we conclude that the restriction of \(\relgxy|_y\) to \(H^0(\fibreXy,\omega_{\fibreXy})^{\pi_1(Y^0)}\) is an isomorphism.

     Finally, observe that by the global invariant cycles theorem, the image of the restriction \(H^0(X,\Omega_X^d)\to H^0(\fibreXy,\omega_{\fibreXy})\) is exactly \(H^0(\fibreXy,\omega_{\fibreXy})^{\pi_1(Y^0)}\), and that this restriction factors through \(H^0(Y,f_{\ast}\omega_{X/Y})\); in fact 
\[
H^0(Y,f_{\ast}\omega_{X/Y})\to H^0(F,\omega_F)^{\pi_1(Y^0)}
\]
is an isomorphism by the same type of argument as in step 2. It follows that \(H^0(\relgxy)\colon H^0(Y,f_{\ast}\omega_{X/Y})\to H^k(Y,f_{\ast}\omega_X)\) is an isomorphism as desired.
\end{proof}

The case where \(Z\) is a point immediately gives the following.
\begin{corollary}
Suppose \(f\colon X\to Y\) is a surjective morphism of relative dimension \(k\) between smooth projective varieties. If \(H^k(Y,f_{\ast}\omega_X)\ne 0\) then \(\omega_{X/Y}\) is effective.
\label{cor:effective-relative-canonical}
\end{corollary}

\subsection{Morphism of VHS}
\label{sec:org427ae75}
 Taking the direct image of \(\relgxzq\) along \(g\) yields 
\[
g_{\ast}\relgxz\colon q_{\ast}\omega_{X/Z}\to g_{\ast}\omega_{Y/Z}\otimes R^kq_{\ast}\omega_{X/Z}.
\]
The goal is to recover this morphism from a map of VHS, at least over the locus where \(g\) and \(q\) are smooth.

Let \(Z^0\subset Z\) be a Zariski-open subset over which \(q\) and \(g\) are smooth, and let \(q^0\colon X^0\to Z^0\) and \(g^0\colon Y^0\to Z^0\) be the corresponding restrictions. Let \(d=\dim X-\dim Y\). We construct a morphism of VHS \(R^kg^0_{\ast}\mathbb{Q}_{Y^0}\otimes R^dq^0_{\ast}\mathbb{Q}_{X^0}\to R^{d+k}q^0_{\ast}\mathbb{Q}_{X^0}\) as follows. A section of \(R^kg^0_{\ast}\mathbb{Q}_{Z^0}\) is locally a cohomology class \(\alpha\in H^k(g^{-1}(U),\mathbb{Q})\) and a section of \(R^dq^0_{\ast}\mathbb{Q}_X\) is locally a class \(\beta\in H^d(q^{-1}(U),\mathbb{Q})\) for small open \(U\subset Z\). Thus we get an element \(f^{\ast}\alpha\wedge\beta\in H^{d+k}(q^{-1}(U),\mathbb{Q})\), which defines a local section of \(R^{d+k}q^0_{\ast}\mathbb{Q}_{X^0}\). As this is compatible with the Hodge filtrations, we get a morphism of VHS.

As \(q^0\) and \(g^0\) are smooth, dualizing gives the desired map \(\relgvhs\colon R^{d+k}q^0_{\ast}\mathbb{Q}_{X^0}\to R^kg^0_{\ast}\mathbb{Q}_{Y^0}\otimes R^{d+2k}q^0_{\ast}\mathbb{Q}_{X^0}\).

\begin{lemma}
Suppose \(g\) is smooth and \(\omega_{Y/Z}\) is trivial. On the lowest graded piece of the Hodge filtration, the morphism
\begin{align*}
R^{d+k}q^0_{\ast}\mathbb{Q}_{X^0}\otimes \mathcal{O}_{Z^0}\to R^kg^0_{\ast}\mathbb{Q}_{Y^0}\otimes R^{d+2k}q^0_{\ast}\mathbb{Q}_{X^0}\otimes \mathcal{O}_{Z^0}
\end{align*}
induced by \(\relgvhs\) agrees with the restriction to \(Z^0\) of 
\begin{align*}
g_{\ast}\relgxzq\colon q_{\ast}\omega_{X/Z}\to g_{\ast}\omega_{Y/Z}\otimes R^kq_{\ast}\omega_{X/Z}.
\end{align*}
\label{lemma:coherent-and-vhs-gysin-maps-agree}
\end{lemma}
\begin{proof}
As \(q^0\) and \(g^0\) are smooth, base change applies to the direct images of the line bundles. By proper base change for the direct images of constant sheaves, it thus suffices to assume that \(Z\) is a point. Then \(g_{\ast}\relgxzq\) is just a map 
\[
H^0(X,\omega_X)\to H^0(Y,\omega_Y)\otimes H^k(X,\omega_X),
\]
and we must identify the dual 
\[
H^k(Y,\mathcal{O}_Y)\otimes H^d(X,\mathcal{O}_X)\to H^{d+k}(X,\mathcal{O}_X)
\]
with the cup product map, by definition of \(\relgvhs\).

By assumption, \(\omega_{Y/Z}=\omega_Y\) is trivial, so fix an isomorphism by choosing a non-zero \(\tau\in H^0(Y,\omega_Y)\). Suppose given 
\[
\alpha\in H^k(Y,\mathcal{O}_Y),\beta\in H^d(X,\mathcal{O}_Y),\gamma\in H^0(X,\omega_X).
\]
Since \(H^k(Y,\mathcal{O}_Y)\) is dual to \(H^0(Y,\omega_Y)\), we can assume that \(\alpha\) is dual to \(\tau\) under Serre duality. The claim is that
 \[
f^{\ast}\alpha\wedge\beta\wedge\gamma=(\alpha\otimes\beta,\relgxzq(\gamma))
\]
in \(H^{k+d}(X,\omega_X)\), where the right hand side is the Serre duality pairing of \(H^k(Y,\mathcal{O}_Y)\otimes H^d(X,\mathcal{O}_X)\) with \(H^0(Y,\omega_Y)\otimes H^k(X,\omega_X)\). To see this, note that triviality of \(\omega_Y\) implies that the natural map 
\[
H^0(Y,f_{\ast}\omega_{X/Y})\otimes H^0(Y,\omega_Y)\to H^0(Y,f_{\ast}\omega_X)=H^0(X,\omega_X)
\]
is an isomorphism. Thus there's a global section \(\psi\in H^0(Y,f_{\ast}\omega_{X/Y})\) such that \(\gamma=\tau\tensor\psi\). At least over the smooth locus of \(Y\), \(\psi\) is nothing but a holomorphic \(d\)-form on \(X\) such that \(\gamma=f^{\ast}\tau\wedge\psi\).

For general \(y\in Y\), let \(\fibreXy=f^{-1}(y)\). Then \(\relgxzq(\gamma)\) is exactly the Gysin morphism of the inclusion \(\fibreXy\subset X\) applied to the restriction \(\psi|_{\fibreXy}\), tensor \(\tau\), by Lemma \ref{lemma:fibres-of-G}. To compute, we note now that 
\begin{align*}
(\alpha\otimes\beta,G(\gamma))&=(\alpha,\tau)\cdot(\beta|_{\fibreXy},\psi|_{\fibreXy})\\
&=(\beta|_{\fibreXy},\psi|_{\fibreXy})
\end{align*}
since \(\alpha\) and \(\tau\) are dual. In particular, \((\beta|_{\fibreXy},\psi|_{\fibreXy})\) is independent of \(y\) (this is related to monodromy invariance of \(\psi|_{\fibreXy}\)). On the other hand, 
\[
f^{\ast}\alpha\wedge\beta\wedge\gamma=f^{\ast}(\alpha\wedge\tau)\wedge\beta\wedge \psi
\]
To integrate the right hand side, we integrate \(\beta\wedge\psi\) over fibres \(\fibreXy\), then integrate \(\alpha\wedge\tau\) over \(Y\); but that gives exactly the desired result.
\end{proof}

Finally, we can state and prove the following more precise version of Theorem \ref{thm:B-simple}.

\begin{theorem}
Suppose \(X\xrightarrow f Y\xrightarrow g Z\) are surjective morphisms of smooth projective varieties, and let \(q=g\circ f\). Suppose further that \(g\) is a smooth fibration and \(\omega_{Y/Z}\) is trivial. Then the morphism \(\relgx\colon g_{\ast}\omega_X\to g^{\ast}R^kg_{\ast}f_{\ast}\omega_X\) is split surjective.
\label{thm:B}
\end{theorem}
\begin{proof}
By generic base change (Proposition \ref{prop:generic-base-change}), we can fix an open \(Z^0\subset Z\) over which \(q\) is smooth and base change to fibres over \(Z\) applies to the sheaves \(f_{\ast}\omega_X\) and \(g^{\ast}R^kg_{\ast}f_{\ast}\omega_X\).

We can in fact assume that \(Z\setminus Z^0\) is a normal crossings divisor. If not, consider a log resolution \(\pi_Z\colon Z'\to Z\) of \(Z\setminus Z^0\); by pullback we get the following diagram, where the vertical maps are birational.

\begin{equation*}
\begin{tikzcd}
X'\ar[r,"f'"]\ar[d,"\pi_X"]&Y'\ar[r,"g'"]\ar[d,"\pi_Y"]&Z'\ar[d,"\pi_Z"]\\
X\ar[r,"f"]&Y\ar[r,"g"]&Z
\end{tikzcd}
\end{equation*}

Assuming \(f'_{\ast}\omega_{X'}\cong g'^{\ast}R^kg'_{\ast}f'_{\ast}\omega_{X'}\oplus\sh Q\), we get \(f_{\ast}\omega_X\cong g^{\ast}R^kg_{\ast}f_{\ast}\omega_X\oplus\pi_{Y\ast}\sh Q\). Indeed 
\begin{align*}
\pi_{Y\ast}f'_{\ast}\omega_{X'}&=f_{\ast}\pi_{X\ast}\omega_{X'}\\
&=f_{\ast}\omega_X
\end{align*}
since \(\pi_X\) is birational, and on the other hand
\begin{align*}
\pi_{Y\ast}g'^{\ast}R^kg'_{\ast}f'_{\ast}\omega_{X'}&=g^{\ast}\pi_{Z\ast}R^kg'_{\ast}f'_{\ast}\omega_{X'}\\
&=g^{\ast}R^kg_{\ast}f_{\ast}\omega_X
\end{align*}
where the first line is by flat base change along \(g\), and the second by Kollár's result (Theorem \ref{thm:kollarII-3.4}).

Assume thus that \(Z\setminus Z^0\) is a normal crossings divisor. In particular, \(q\) is smooth away from a normal crossings divisor, which implies that \(R^kq_{\ast}\omega_X\) and its direct summand \(R^kg_{\ast}f_{\ast}\omega_X\) are locally free by Kollár's result (Theorem \ref{thm:kollarII-2.6}). Consider then 
\[
g_{\ast}\relgxz\colon q_{\ast}\omega_{X/Z}\to g_{\ast}\omega_{Y/Z}\otimes R^kq_{\ast}\omega_{X/Z}
\]
which, by Lemma \ref{lemma:coherent-and-vhs-gysin-maps-agree}, is induced over \(Z^0\) by a morphism of VHS 
\[
\relgvhs\colon R^{d+k}q^0_{\ast}\mathbb{Q}_{X^0}\to R^kg^0_{\ast}\mathbb{Q}_{Y^0}\otimes R^{d+2k}q^0_{\ast}\mathbb{Q}_{X^0}.
\]
Since the category of polarizable VHS is semisimple \cite[Theorem 10.13]{Peters2008}, there is a direct sum decomposition of VHS \(R^{d+k}q^0_{\ast}\mathbb{Q}_{X^0}\cong I\oplus K\) where \(K\) is the kernel of \(\Phi\), and \(I\) maps isomorphically to the image of \(\Phi\). We also have a decomposition \(R^kg^0_{\ast}\mathbb{Q}_{Z^0}\otimes R^{d+2k}q^0_{\ast}\mathbb{Q}_{X^0}\cong I\oplus C\), and the resulting \(I\oplus K\to I\oplus C\) is just the identity map on \(I\) while vanishing on \(K\).

Again by Theorem \ref{thm:kollarII-2.6}, the lowest piece of the Hodge filtration of the upper canonical extension of \(R^{d+k}q_{\ast}^0\mathbb{Q}_{X^0}\) is exactly \(q_{\ast}\omega_{X/Z}\), while the same construction applied to \(R^kg^0_{\ast}\mathbb{Q}_{Y^0}\otimes R^{d+2k}q^0_{\ast}\mathbb{Q}_{X^0}\) yields \(g_{\ast}\omega_{Y/Z}\otimes R^kq_{\ast}\omega_{X/Z}\). Note in the latter case that \(R^kg^0_{\ast}\mathbb{Q}_{Y^0}\) has trivial monodromy around the complement of \(Z^0\), since \(g\) is smooth, so taking canonical extensions and tensor products does in fact commute in this case by the discussion in section \ref{sec:org27b505d}.

Let \(\sh I,\sh K,\sh C\) be the lowest pieces of the Hodge filtration on the upper canonical extensions of \(I,K\) and \(C\) respectively. The formation of canonical extensions is compatible with direct sums, so \(q_{\ast}\omega_{X/Z}\cong\sh I\oplus \sh K\). By Lemma \ref{lemma:coherent-and-vhs-gysin-maps-agree}, the image of \(g_{\ast}\relgxzq\) inside \(g_{\ast}\omega_{Y/Z}\otimes R^kq_{\ast}\omega_{X/Z}\) agrees with \(\sh I\) over \(Z^0\). Since all sheaves involved are locally free, it follows that \(\sh I\) is in fact the image of \(g_{\ast}\relgxzq\).

Back on \(Y\), the push-pull adjunction for \(g\) applied to \(\relgxzq\), together with the projection formula, gives the following commuting diagram.

\begin{equation*}
     \begin{tikzcd}
     g^{\ast}q_{\ast}\omega_{X/Z}\ar[r,"g^{\ast}g_{\ast}\relgxzq"]\ar[d]& g^{\ast}g_{\ast}\omega_{Y/Z}\otimes g^{\ast}R^kq_{\ast}\omega_{X/Z}\ar[d]\\
     f_{\ast}\omega_{X/Z}\ar[r,"\relgxzq"]&\omega_{Y/Z}\otimes g^{\ast}R^kq_{\ast}\omega_{X/Z}
     \end{tikzcd}
\end{equation*}

As \(\omega_{Y/Z}\) is trivial and \(g\) is a fibration, the right side vertical map is an isomorphism. Moreover, \(g^{\ast}\sh I\) is a direct summand of both the top left and bottom right corners, and the composition through \(f_{\ast}\omega_{X/Z}\), when restricted to \(g^{\ast}\sh I\), is the identity. It thus remain only to show that \(\sh I=R^kg_{\ast}f_{\ast}\omega_{X/Z}\), as it would then follow that \(g^\ast\sh I=g^\ast R^k g_\ast f_\ast\omega_{X/Z}\) is the image of \(\relgxzq\), hence also of \(\relgxz\), and the previous diagram yields a splitting of \(\relgxz\) as desired.

Thus we must show that \(\relgxz\) remains surjective after pushing forward to \(Z\).
On \(Z\), \(\sh I\) and \(R^kg_{\ast}f_{\ast}\omega_{X/Z}\) are both locally free subsheaves of \(R^kq_{\ast}\omega_{X/Z}\) (in fact direct summands). Thus it suffices to show that for general \(z\in Z\), the fibre of \(g_{\ast}\relgxz\) at \(z\) is surjective. By generic base change (Proposition \ref{prop:generic-base-change}) we can assume that \(Z\) is just a point, in which case we are to show that the map induced by \(\relgxz\) on global sections is surjective.

Since \(\relgxz\) and \(\relgxy\) are related by twisting by \(\omega_{Y/Z}\), fixing a trivialization of \(\omega_{Y/Z}\) identifies the two maps, so we are done by Lemma \ref{lemma:fibres-of-G}.
\end{proof}

It seems more natural to consider VHS on \(Y\) rather than \(Z\) to get the splitting, but there's a technical issue with that approach. Namely, one ends up having to take a resolution \(\pi\colon Y'\to Y\) of \(Y\) that doesn't come from a resolution of \(Z\) by pullback. One then wants to express \(\pi^{\ast}g^{\ast}R^kg_{\ast}f_{\ast}\omega_{X/Z}\) as a direct summand of the canonical extension of a VHS pulled back from an open subset of \(Z\), with the hope of splitting \(\relgxy\). While \(g\) is smooth, so functoriality of canonical extensions is not an issue there, the composition \(\pi\circ g\) is not smooth, so it's not clear what the canonical extension on \(Y'\) gives. This functoriality issue can be fixed by appealing to Hodge modules, which would give a proof along these lines even without the assumptions on \(g\) and \(\omega_{Y/Z}\).

\subsection{Effectiveness of relative canonical bundles and fibres of the Albanese morphism}
\label{sec:orgf247e64}
Corollary \ref{cor:effective-relative-canonical} can be used to give a variation of a proof of a theorem by Jiang. For a smooth projective variety \(X\), let \(P_n=h^0(X,\omega_X^n)\) denote the plurigenera of \(X\).

\begin{theorem}[{\cite[Theorem 3.1]{Jiang2009}}]
Suppose \(X\) be a smooth projective variety with \(P_1(X)=P_2(X)=1\). Then the fibres of the Albanese mapping are connected.
\end{theorem}
\begin{proof}
By \cite{Hacon2002}, it is known that the Albanese mapping \(a_X\colon X\to \operatorname{Alb}(X)\) is surjective in this case. Taking the Stein factorization and resolving singularities of the middle term (replacing \(X\) with a birational modification) yields a factorization \(X\xrightarrow g V\xrightarrow b \operatorname{Alb}(X)\) of \(a_X\) where \(b\) is generically finite. It is a theorem of Chen and Hacon \cite{Chen2001} that if, in this case, \(P_1(V)=P_2(V)=1\), then the Albanese mapping of \(V\) is birational, and as a consequence also \(b\) is birational so \(a_X\) has connected fibres. To prove the theorem it thus suffices to show that \(\omega_{X/V}\) is effective.

However, it follows from the theory of GV-sheaves, by an observation of Ein and Lazarsfeld \cite{Ein1997}, that \(H^g(\operatorname{Alb}(X),a_{X\ast}\omega_X)\ne 0\) since \(P_1(X)=P_2(X)=1\), where \(g=\dim \operatorname{Alb}(X)=\dim V\). Since \(b\) is generically finite, and by Kollár's result (Theorem \ref{thm:kollarII-3.4}), \(H^g(\operatorname{Alb}(X),a_{X\ast}\omega_X)=H^g(V,g_{\ast}\omega_X)\). Then Corollary \ref{cor:effective-relative-canonical} gives the conclusion.
\end{proof}

\sloppy
\printbibliography

@Misc{Schnell2019,
  author      = {Christian Schnell},
  title       = {The Fourier-Mukai transform made easy},
  date        = {2019-05-30},
  eprint      = {1905.13287v1},
  eprintclass = {math.AG},
  eprinttype  = {arXiv},
}

@Article{Hacon2004,
  author   = {Christopher D. Hacon},
  title    = {A derived category approach to generic vanishing},
  journal  = {Journal für die reine und angewandte Mathematik (Crelles Journal)},
  year     = {2004},
  volume   = {575},
  month    = {10},
  pages    = {173-187},
}

@InCollection{Pareschi2011,
  author     = {Pareschi, Giuseppe and Popa, Mihnea},
  title      = {Regularity on abelian varieties {III}: relationship with generic vanishing and applications},
  booktitle  = {Grassmannians, moduli spaces and vector bundles},
  year       = {2011},
  volume     = {14},
  series     = {Clay Math. Proc.},
  publisher  = {Amer. Math. Soc., Providence, RI},
  pages      = {141--167},
  mrclass    = {14K05 (14F05 14F17)},
  mrnumber   = {2807853},
  mrreviewer = {Nicolae Manolache},
}

@Article{Green1987,
  author    = {Mark Green and Robert Lazarsfeld},
  title     = {Deformation theory, generic vanishing theorems, and some conjectures of Enriques, Catanese and Beauville},
  journal   = {Inventiones Mathematicae},
  year      = {1987},
  volume    = {90},
  number    = {2},
  month     = jun,
  pages     = {389--407},
  doi       = {10.1007/bf01388711},
  publisher = {Springer Science and Business Media {LLC}},
}

@Article{Kollar1986a,
  author    = {János Kollár},
  title     = {Higher Direct Images of Dualizing Sheaves II},
  journal   = {Annals of Mathematics},
  year      = {1986},
  volume    = {124},
  number    = {1},
  pages     = {171--202},
  issn      = {0003486X},
  doi       = {10.2307/1971390},
  publisher = {Annals of Mathematics},
}

@Book{Peters2008,
  title={Mixed Hodge Structures},
  author={Peters, C.A.M. and Steenbrink, J.H.M.},
  isbn={9783540770176},
  lccn={2007942592},
  series={Ergebnisse der Mathematik und ihrer Grenzgebiete. 3. Folge / A Series of Modern Surveys in Mathematics},
  year={2008},
  publisher={Springer Berlin Heidelberg},
}

@Article{Lombardi2017,
  author    = {Luigi Lombardi and Mihnea Popa and Christian Schnell},
  title     = {Pushforwards of pluricanonical bundles under morphisms to abelian varieties},
  journal   = {Journal of the European Mathematical Society},
  year      = {2020},
  volume    = {22},
  number    = {8},
  month     = may,
  pages     = {2511--2536},
  doi       = {10.4171/jems/970},
  publisher = {European Mathematical Society Publishing House},
}

@Article{Chen2018,
  author    = {Jungkai Alfred Chen and Zhi Jiang},
  title     = {Positivity in varieties of maximal Albanese dimension},
  journal   = {Journal f\"{u}r die reine und angewandte Mathematik (Crelles Journal)},
  year      = {2018},
  volume    = {2018},
  number    = {736},
  month     = mar,
  pages     = {225--253},
  doi       = {10.1515/crelle-2015-0027},
  publisher = {Walter de Gruyter {GmbH}},
}

@Article{Kollar1986,
  author    = {János Kollár},
  title     = {Higher Direct Images of Dualizing Sheaves I},
  journal   = {Annals of Mathematics},
  year      = {1986},
  volume    = {123},
  number    = {1},
  pages     = {11--42},
  issn      = {0003486X},
  doi       = {10.2307/1971351},
  publisher = {Annals of Mathematics},
}

@Article{Green1991,
  author    = {Mark Green and Robert Lazarsfeld},
  title     = {Higher Obstructions to Deforming Cohomology Groups of Line Bundles},
  journal   = {Journal of the American Mathematical Society},
  year      = {1991},
  volume    = {4},
  number    = {1},
  pages     = {87--103},
  issn      = {08940347, 10886834},
  url       = {http://www.jstor.org/stable/2939255},
  publisher = {American Mathematical Society},
}

@Article{Simpson1993,
  author    = {Simpson, Carlos},
  title     = {Subspaces of moduli spaces of rank one local systems},
  journal   = {Annales scientifiques de l'\'Ecole Normale Sup\'erieure},
  year      = {1993},
  language  = {en},
  volume    = {Ser. 4, 26},
  number    = {3},
  pages     = {361-401},
  doi       = {10.24033/asens.1675},
  mrnumber  = {94f:14008},
  publisher = {Elsevier},
  zbl       = {0798.14005},
}

@Article{Pareschi2015,
  author    = {Giuseppe Pareschi and Mihnea Popa and Christian Schnell},
  title     = {Hodge modules on complex tori and generic vanishing for compact K\"{a}hler manifolds},
  journal   = {Geometry {\&} Topology},
  year      = {2017},
  volume    = {21},
  number    = {4},
  month     = may,
  pages     = {2419--2460},
  doi       = {10.2140/gt.2017.21.2419},
  publisher = {Mathematical Sciences Publishers},
}

@Misc{Jiang2020,
  author      = {Jiang, Zhi},
  title       = {M-regular Decompositions for pushforwards of pluricanonical bundles of pairs to abelian varieties},
  date        = {2020-06-09},
  eprint      = {2006.02393v2},
  eprintclass = {math.AG},
  eprinttype  = {arXiv},
  comment     = {The references and the acknowledgements have been corrected and updated},
}

@Misc{Meng2020,
  author        = {Fanjun Meng},
  title         = {Pushforwards of klt pairs under morphisms to abelian varieties},
  date          = {2020-05-28},
  eprint        = {2005.13761v1},
  eprintclass   = {math.AG},
  eprinttype    = {arXiv},
}

@Article{Jiang2009,
  author      = {Jiang, Zhi},
  title       = {An effective version of a theorem of Kawamata on the Albanese map},
  journal     = {Communications in Contemporary Mathematics},
  year        = {2011},
  volume      = {13},
  number      = {03},
  month       = jun,
  pages       = {509--532},
  doi         = {10.1142/s0219199711004397},
  publisher   = {World Scientific Pub Co Pte Lt},
}

@Article{Chen2001,
  author    = {Jungkai A. Chen and Christopher D. Hacon},
  title     = {Characterization of abelian varieties},
  journal   = {Inventiones Mathematicae},
  year      = {2001},
  volume    = {143},
  number    = {2},
  month     = feb,
  pages     = {435--447},
  doi       = {10.1007/s002220000111},
  publisher = {Springer Science and Business Media {LLC}},
}

@Article{Hacon2002,
  author    = {Christopher D. Hacon and Rita Pardini},
  title     = {On the birational geometry of varieties of maximal Albanese dimension},
  journal   = {Journal f\"{u}r die reine und angewandte Mathematik (Crelles Journal)},
  year      = {2002},
  volume    = {2002},
  number    = {546},
  month     = jan,
  doi       = {10.1515/crll.2002.042},
  publisher = {Walter de Gruyter {GmbH}},
}

@Article{Ein1997,
  author    = {Lawrence Ein and Robert Lazarsfeld},
  title     = {Singularities of theta divisors and the birational geometry of irregular varieties},
  journal   = {Journal of the American Mathematical Society},
  year      = {1997},
  volume    = {10},
  number    = {1},
  pages     = {243--258},
  doi       = {10.1090/s0894-0347-97-00223-3},
  publisher = {American Mathematical Society ({AMS})},
}
\end{document}